\documentclass[11pt]{article}

\usepackage{amsmath,amssymb,amsthm}
\usepackage[a4paper,textwidth=15cm,top=3cm,bottom=3cm]{geometry}
\usepackage{url}
\newcommand{\doi}[1]{doi: \url{#1}}

\newtheorem{theorem}{Theorem}[section]
\newtheorem{proposition}[theorem]{Proposition}
\newtheorem{lemma}[theorem]{Lemma}
\newtheorem{corollary}[theorem]{Corollary}
\newtheorem{question}[theorem]{Question}
\newtheorem{example}[theorem]{Example}

\theoremstyle{definition}
\newtheorem{definition}[theorem]{Definition}

\title{Raikov Remainders of Quotients by Raikov-Complete Almost Metrizable Normal Subgroups}
\newcommand{\shorttitle}{Raikov Quotient Remainders}

\author{Xing-Yu Hu\thanks{Corresponding author.}\\
{\small School of Mathematics and Statistics, Hanjiang Normal University}\\
{\small Shiyan, Hubei 442000, China}\\
{\small \texttt{huxingyu@hjnu.edu.cn}}}

\date{}

\begin{document}

\pagestyle{myheadings}
\markboth{X.-Y. Hu}{\shorttitle}

\maketitle

\begin{abstract}
Let \(N\) be a closed normal subgroup of a topological group \(G\), and let \(\widehat q:\rho G\to\rho(G/N)\) extend the quotient homomorphism. We prove that
\[
        r_\rho(G/N)=\widehat q\bigl(r_\rho(G)\bigr)
\]
holds if and only if \(\widehat q\) is onto and \(\widehat q^{-1}(G/N)=G\), and we show that neither condition implies the other. Both conditions hold whenever \(N\) is Raikov complete and almost metrizable. Consequently, pseudocompactness of \(r_\rho(G)\) passes to \(r_\rho(G/N)\). In particular, the conclusion applies to closed locally compact normal subgroups.
\end{abstract}



\section{Introduction}

The study of remainders of topological groups is a well-established part of general topology. Comfort and Ross studied pseudocompactness in topological groups \cite{ComfortRoss1966}. Arhangel'skii proved that, for every Hausdorff compactification \(bG\) of a topological group \(G\), the remainder \(bG\setminus G\) is either pseudocompact or Lindel\"of \cite{Arhangelskii2008}. Subsequent work addressed normality, first countability, covering properties, and metrizability of compactification remainders \cite{Arhangelskii2005Remainders,Arhangelskii2009,ArhangelskiiVanMill2016,ArhangelskiiVanMill2017}. Arhangel'skii and Choban studied group-remainders, including the Rajkov remainder \cite{ArhangelskiiChoban2018}, and later considered completeness-type properties, products, and perfect images of group-remainders \cite{ArhangelskiiChoban2019,ArhangelskiiChoban2021}.

For a topological group \(G\), let \(\rho G\) denote its Raikov completion and put
\[
        r_\rho(G)=\rho G\setminus G.
\]
This is the Raikov remainder of \(G\). It is a completion remainder and need not be a compactification remainder. We study its behaviour under the quotient homomorphism
\[
        q:G\longrightarrow G/N,
\]
where \(N\) is a closed normal subgroup of \(G\). The homomorphism \(q\) extends uniquely to
\[
        \widehat q:\rho G\longrightarrow\rho(G/N).
\]
Although \(q\) is onto, \(\widehat q\) need not be onto. Even when it is onto, a point of \(r_\rho(G)\) may be mapped into \(G/N\). These are the two obstructions to the quotient-remainder identity.

Our first quotient result gives an exact criterion. For an arbitrary closed normal subgroup \(N\),
\[
        r_\rho(G/N)=\widehat q\bigl(r_\rho(G)\bigr)
\]
holds precisely when
\[
        \widehat q(\rho G)=\rho(G/N)
        \quad\hbox{and}\quad
        \widehat q^{-1}(G/N)=G.
\]
The examples in the final subsection show that neither condition implies the other. Thus surjectivity of the map between completions and the inverse-image condition are distinct requirements.

The main theorem verifies both requirements when \(N\) is Raikov complete and almost metrizable. In this case \(\widehat q\) is a quotient homomorphism with kernel \(N\),
\[
        \widehat q^{-1}(G/N)=G,
\]
and
\[
        r_\rho(G/N)=\widehat q\bigl(r_\rho(G)\bigr).
\]
It follows that pseudocompactness of \(r_\rho(G)\) passes to \(r_\rho(G/N)\). Closed locally compact normal subgroups form an important special case.

The proof separates the remainder criterion from the completeness of the quotient \((\rho G)/N\). Related completion facts occur in the proof of \cite[Theorem~4.5]{MoralesTkachenko2016}. For almost metrizable kernels, the required quotient completeness follows from Leischner's theorem on neutral subgroups \cite{Leischner1993}. The relevant completion results are discussed before Proposition~\ref{prop:general-criterion} and Lemma~\ref{lem:am-quotient-complete}.

For comparison, we also give formulas for closed subgroups and finite products and derive the corresponding preservation results. Section 2 fixes the notation and the preservation facts used later. Section 3 treats closed subgroups, finite products, quotient homomorphisms, the general completion criterion, Raikov-complete almost metrizable kernels, applications, and sharpness questions.

\section{Preliminaries}\label{sec:preliminaries}

All topological groups are assumed to be Hausdorff. The remainders and subspaces considered below are therefore Tychonoff. We regard the empty space as pseudocompact.

For background on general topology, including countable compactness, see Engelking \cite{Engelking1989}. For the function-theoretic background on pseudocompactness, see Hewitt \cite{Hewitt1948} and Gillman and Jerison \cite{GillmanJerison1976}. For Raikov completions and uniform structures on topological groups, see \cite{ArhangelskiiTkachenko2008,RoelckeDierolf1981}. For a topological group \(G\), let \(\rho G\) denote its Raikov completion. We regard \(G\) as its image in \(\rho G\). Thus \(G\) is a dense subgroup of \(\rho G\). A topological group is Raikov complete if the natural homomorphism into its Raikov completion is onto. A topological group \(L\) is almost metrizable if it contains a compact subgroup \(K\) such that the coset space \(L/K\) is metrizable. Equivalently, \(L\) contains a compact subgroup \(K\) with a countable neighbourhood base in \(L\) \cite[Lemma~4.3.19]{ArhangelskiiTkachenko2008}.

A subgroup \(H\) of a topological group \(L\) is \emph{neutral} if, for every identity neighbourhood \(U\), there is an identity neighbourhood \(V\) such that \(VH\subseteq HU\) \cite[Definition~5.29]{RoelckeDierolf1981}. Normal subgroups are neutral because \(UH=HU\). A group is \emph{inframetrizable} if it embeds topologically in a \v{C}ech-complete group \cite[Proposition--Definition~11.16]{RoelckeDierolf1981}.

\begin{definition}[{\cite[p.~82]{ArhangelskiiChoban2018}}]
The Raikov remainder of a topological group \(G\) is
\[
        r_\rho(G)=\rho G\setminus G.
\]
\end{definition}

Thus \(r_\rho(G)\) is the group-remainder associated with the Raikov completion.

The spelling Rajkov also appears in the literature. We use the spelling Raikov, while preserving the original spelling in the titles of cited papers.

We shall use the following extension property of Raikov completions.

\begin{lemma}\label{lem:extension-to-completions}
Let \(f:G\to K\) be a continuous homomorphism of topological groups. Then \(f\) is uniformly continuous with respect to the two-sided uniformities and admits a unique continuous homomorphic extension
\[
        \widehat f:\rho G\to \rho K.
\]
If \(f(G)\) is dense in \(K\), then \(\widehat f(\rho G)\) is dense in \(\rho K\).
\end{lemma}

\begin{proof}
By \cite[Proposition~1.8.12]{ArhangelskiiTkachenko2008}, \(f\) is uniformly continuous for the two-sided uniformities. It therefore extends uniquely to a continuous homomorphism \(\widehat f:\rho G\to\rho K\) by \cite[Corollary~3.6.17]{ArhangelskiiTkachenko2008}. If \(f(G)\) is dense in \(K\), then \(\widehat f(\rho G)\) contains \(f(G)\) and is therefore dense in \(\rho K\).
\end{proof}

We shall use the following standard preservation facts.

\begin{lemma}[{\cite[Theorem~3.10.24]{Engelking1989}}]\label{lem:continuous-image}
A continuous image of a pseudocompact space is pseudocompact.
\end{lemma}

The next lemma is the consequence needed for closed-subgroup remainders.

\begin{lemma}\label{lem:closed-normal-pseudo}
Every closed subspace of a normal pseudocompact space is pseudocompact.
\end{lemma}

\begin{proof}
Every normal pseudocompact space is countably compact \cite[Theorem~3.10.21]{Engelking1989}. Its closed subspaces are countably compact \cite[Theorem~3.10.4]{Engelking1989}, and every countably compact Tychonoff space is pseudocompact \cite[Theorem~3.10.20]{Engelking1989}.
\end{proof}

This explains the normality assumption in the closed-subgroup corollary below. The closed-subgroup formula alone does not suffice, since pseudocompactness need not pass to arbitrary closed subspaces.

\section{Main results}\label{sec:main-results}

\subsection{Closed subgroups}\label{sec:closed-subgroups}

We begin with the standard closed-subgroup formula for Raikov completions \cite[Theorem~3.6.14 and Exercise~3.6.m]{ArhangelskiiTkachenko2008}.

\begin{proposition}\label{prop:closed-subgroup-formula}
Let \(H\) be a closed subgroup of a topological group \(G\). Then the natural embedding of \(H\) into \(G\) extends to a topological isomorphism from \(\rho H\) onto \(\overline H^{\rho G}\). Via this isomorphism,
\[
        r_\rho(H)
        =
        \overline H^{\rho G}\cap r_\rho(G).
\]
\end{proposition}

\begin{proof}
By \cite[Theorem~3.6.14 and Exercise~3.6.m]{ArhangelskiiTkachenko2008}, the inclusion \(H\hookrightarrow G\) extends to a topological isomorphism from \(\rho H\) onto \(\overline H^{\rho G}\). We henceforth regard \(\rho H\) as \(\overline H^{\rho G}\) via this isomorphism.

Since \(H\) is closed in \(G\), one has
\[
        \overline H^{\rho G}\cap G=H.
\]
Therefore
\[
\begin{aligned}
        r_\rho(H)
        &=
        \overline H^{\rho G}\setminus H  \\
        &=
        \overline H^{\rho G}\setminus
        \bigl(\overline H^{\rho G}\cap G\bigr)  \\
        &=
        \overline H^{\rho G}\cap(\rho G\setminus G)  \\
        &=
        \overline H^{\rho G}\cap r_\rho(G).
\end{aligned}
\]
This proves the formula.
\end{proof}

\begin{corollary}
Let \(H\) be a closed subgroup of a topological group \(G\). If \(r_\rho(G)\) is normal as a topological space and pseudocompact, then \(r_\rho(H)\) is pseudocompact.
\end{corollary}

\begin{proof}
By Proposition~\ref{prop:closed-subgroup-formula}, \(r_\rho(H)\) is a closed subspace of \(r_\rho(G)\). The conclusion follows from Lemma~\ref{lem:closed-normal-pseudo}.
\end{proof}

\begin{corollary}
Let \(H\) be a closed subgroup of a topological group \(G\). If \(r_\rho(G)\) is compact, then \(r_\rho(H)\) is compact. If \(r_\rho(G)\) is countably compact, then \(r_\rho(H)\) is countably compact.
\end{corollary}

\begin{proof}
The formula
\[
        r_\rho(H)
        =
        \overline H^{\rho G}\cap r_\rho(G)
\]
shows that \(r_\rho(H)\) is closed in \(r_\rho(G)\). Compactness and countable compactness pass to closed subspaces.
\end{proof}

\subsection{Finite products}\label{sec:finite-products}

For finite products, we use the following consequence of the general product theorem for Raikov completions.

\begin{lemma}\label{lem:product-completion}
For topological groups \(G\) and \(K\), the natural homomorphism
\[
        \rho(G\times K)\longrightarrow \rho G\times \rho K
\]
is a topological isomorphism.
\end{lemma}

\begin{proof}
Apply \cite[Corollary~3.6.23]{ArhangelskiiTkachenko2008} to the family indexed by a two-point set, with factors \(G\) and \(K\). The resulting product isomorphism is the displayed natural homomorphism.
\end{proof}

\begin{proposition}\label{prop:product-formula}
For topological groups \(G\) and \(K\),
\[
        r_\rho(G\times K)
        =
        \bigl(r_\rho(G)\times \rho K\bigr)
        \cup
        \bigl(\rho G\times r_\rho(K)\bigr),
\]
where the equality is taken under the natural topological isomorphism
\[
        \rho(G\times K)\cong\rho G\times\rho K.
\]
\end{proposition}

\begin{proof}
By Lemma~\ref{lem:product-completion}, one has
\[
\begin{aligned}
        r_\rho(G\times K)
        &=
        \rho(G\times K)\setminus(G\times K)\\
        &=
        (\rho G\times\rho K)\setminus(G\times K).
\end{aligned}
\]
A point \((x,y)\in\rho G\times\rho K\) lies outside \(G\times K\) precisely when \(x\notin G\) or \(y\notin K\). Thus the complement of \(G\times K\) in \(\rho G\times\rho K\) is exactly
\[
        \bigl(r_\rho(G)\times \rho K\bigr)
        \cup
        \bigl(\rho G\times r_\rho(K)\bigr),
\]
which proves the formula.
\end{proof}

\begin{lemma}[{\cite[Corollary~3.10.27]{Engelking1989}}]\label{lem:pseudo-times-compact}
If \(X\) is pseudocompact and \(C\) is compact, then \(X\times C\) is pseudocompact.
\end{lemma}

\begin{corollary}\label{cor:two-product-precompact}
Let \(G\) and \(K\) be precompact topological groups. If \(r_\rho(G)\) and \(r_\rho(K)\) are pseudocompact, then \(r_\rho(G\times K)\) is pseudocompact.
\end{corollary}

\begin{proof}
Since \(G\) and \(K\) are precompact, \(\rho G\) and \(\rho K\) are compact. By Proposition~\ref{prop:product-formula},
\[
        r_\rho(G\times K)
        =
        \bigl(r_\rho(G)\times \rho K\bigr)
        \cup
        \bigl(\rho G\times r_\rho(K)\bigr).
\]
By Lemma~\ref{lem:pseudo-times-compact}, each set on the right is pseudocompact. Every continuous real-valued function on their union is bounded on each of the two sets and hence on the union. Therefore \(r_\rho(G\times K)\) is pseudocompact.
\end{proof}

\begin{corollary}
Let \(n\) be a positive integer, and let \(G_1,\ldots,G_n\) be precompact topological groups. If \(r_\rho(G_i)\) is pseudocompact for every \(i\in\{1,\ldots,n\}\), then
\[
        r_\rho(G_1\times\cdots\times G_n)
\]
is pseudocompact.
\end{corollary}

\begin{proof}
When \(n=1\), the pseudocompactness of \(r_\rho(G_1)\) is part of the hypothesis. Suppose the assertion holds for \(n-1\), and put
\[
        P=G_1\times\cdots\times G_{n-1}.
\]
Then \(P\) is precompact and \(r_\rho(P)\) is pseudocompact by the induction hypothesis. Applying Corollary~\ref{cor:two-product-precompact} to \(P\) and \(G_n\) completes the induction.
\end{proof}

For finite products of precompact groups, pseudocompactness therefore follows from the product formula and the fact that products with compact spaces preserve pseudocompactness. By contrast, for a quotient homomorphism, the induced map between the corresponding Raikov completions need not be surjective and may map points of the Raikov remainder of the domain group into the quotient group itself.

\subsection{Quotient homomorphisms}\label{sec:quotients}

Let \(G\) be a topological group, let \(N\) be a closed normal subgroup of \(G\), and let
\[
        q:G\longrightarrow G/N
\]
be the quotient homomorphism. By Lemma~\ref{lem:extension-to-completions}, \(q\) induces a continuous homomorphism
\[
        \widehat q:\rho G\longrightarrow \rho(G/N).
\]
Since \(q\) is onto, \(\widehat q\) has dense image in \(\rho(G/N)\), but it need not be onto. A second issue is whether a point of \(r_\rho(G)\) can map into \(G/N\). The inclusion
\[
        G\subseteq \widehat q^{-1}(G/N)
\]
always holds. Equality is equivalent to saying that
\(\widehat q\bigl(r_\rho(G)\bigr)\cap(G/N)=\varnothing\).

\begin{proposition}\label{prop:quotient-criterion}
Let \(N\) be a closed normal subgroup of a topological group \(G\). Let \(q:G\to G/N\) be the quotient homomorphism, and let
\[
        \widehat q:\rho G\to\rho(G/N)
\]
be its extension. Then
\[
        r_\rho(G/N)=\widehat q\bigl(r_\rho(G)\bigr)
\]
if and only if \(\widehat q\) is onto and
\[
        \widehat q^{-1}(G/N)=G.
\]
Under these two conditions, pseudocompactness of \(r_\rho(G)\) implies pseudocompactness of \(r_\rho(G/N)\).
\end{proposition}

\begin{proof}
Suppose first that \(\widehat q\) is onto and that
\[
        \widehat q^{-1}(G/N)=G.
\]
Let \(y\in r_\rho(G/N)\). Choose \(x\in\rho G\) with \(\widehat q(x)=y\). Since \(y\notin G/N\), one has \(x\notin G\), and hence
\[
        y\in\widehat q\bigl(r_\rho(G)\bigr).
\]
Conversely, if \(x\in r_\rho(G)\), then the preimage condition gives \(\widehat q(x)\notin G/N\). Therefore
\[
        \widehat q\bigl(r_\rho(G)\bigr)\subseteq r_\rho(G/N),
\]
and the required equality follows.

Now suppose that
\[
        r_\rho(G/N)=\widehat q\bigl(r_\rho(G)\bigr).
\]
Since \(\widehat q(G)=G/N\), we have
\[
\begin{aligned}
        \widehat q(\rho G)
        &=\widehat q(G)\cup\widehat q\bigl(r_\rho(G)\bigr)\\
        &=(G/N)\cup r_\rho(G/N)\\
        &=\rho(G/N).
\end{aligned}
\]
Thus \(\widehat q\) is onto. If
\(x\in\widehat q^{-1}(G/N)\setminus G\), then \(x\in r_\rho(G)\), so
\[
        \widehat q(x)\in\widehat q\bigl(r_\rho(G)\bigr)=r_\rho(G/N),
\]
contrary to \(\widehat q(x)\in G/N\). Hence
\(\widehat q^{-1}(G/N)=G\).

The last assertion follows from Lemma~\ref{lem:continuous-image}.
\end{proof}

The proposition shows that both conditions are necessary. Dense image does not imply surjectivity, and surjectivity alone does not prevent points of \(\rho G\setminus G\) from mapping into \(G/N\).

\subsection{A general completion criterion}\label{sec:criterion}

The quotient argument is based on the following completion criterion. Closely related completion facts occur in the proof of \cite[Theorem~4.5]{MoralesTkachenko2016}. Under the hypotheses of that theorem, Morales and Tkachenko note, using the \v{C}ech-completeness of \(H\), that \((\rho X)/H\) is the Raikov completion of \(X/H\). In the direction where \(X/H\) is assumed to be strongly Dieudonn\'e complete, they also use
\[
        X=\varphi^{-1}(X/H)
\]
for the quotient homomorphism
\(\varphi:\rho X\to(\rho X)/H\).

In the abelian case, Bello, Chasco, Dom\'inguez, and Tkachenko show that the Raikov completion of a quotient is naturally isomorphic to the corresponding quotient of completions when the Raikov completion of the closed subgroup is \v{C}ech-complete \cite[Proposition~3.9]{BelloChascoDominguezTkachenko2016}. Our quotient-remainder criterion determines exactly when \(\widehat q\) carries \(r_\rho(G)\) onto \(r_\rho(G/N)\). Surjectivity and the inverse-image condition are separate requirements. The examples below show that neither requirement implies the other, while the main theorem verifies both for Raikov-complete almost metrizable kernels.

\begin{lemma}\label{lem:complete-kernel-closed-in-completion}
Let \(L\) be a Hausdorff topological group containing \(G\) as a dense subgroup, and let \(N\) be a Raikov-complete normal subgroup of \(G\). Then \(N\) is a closed normal subgroup of \(L\).
\end{lemma}

\begin{proof}
By \cite[Proposition~1.8.4]{ArhangelskiiTkachenko2008}, the two-sided uniformity of a subgroup is the restriction of the two-sided uniformity of the ambient group. Hence the two-sided uniformity of \(N\) is the one induced from \(L\). Let \(x\in\overline N^{L}\), and choose a net \((n_i)\) in \(N\) that converges to \(x\) in \(L\). This net is Cauchy in \(N\). Since \(N\) is Raikov complete, it converges in \(N\) to some \(n\in N\). As a net in the Hausdorff group \(L\), it converges to both \(x\) and \(n\). Thus \(x=n\), and \(N\) is closed in \(L\).

It remains to prove that \(N\) is normal in \(L\). Let \(x\in L\) and \(n\in N\). Choose a net \((g_i)\) in \(G\) converging to \(x\). Then
\[
        g_i n g_i^{-1}\in N
\]
for every \(i\), because \(N\) is normal in \(G\). The net \((g_i n g_i^{-1})\) converges to \(xnx^{-1}\). Since \(N\) is closed in \(L\), one has \(xnx^{-1}\in N\). Thus \(xNx^{-1}\subseteq N\). Applying the same argument to \(x^{-1}\) gives equality.
\end{proof}

\begin{lemma}\label{lem:subgroup-quotient-embedding}
Let \(L\) be a topological group, let \(K\) be a normal subgroup of \(L\), and let \(G\) be a subgroup of \(L\) containing \(K\). Let
\[
        \pi:L\to L/K
\]
be the quotient map. Then the natural homomorphism
\[
        G/K\longrightarrow \pi(G)
\]
is a topological isomorphism, where \(\pi(G)\) carries the subspace topology inherited from \(L/K\).
\end{lemma}

\begin{proof}
The natural homomorphism is an algebraic isomorphism and is continuous by the definition of the quotient topology on \(G/K\). To prove that it is open, let \(U\) be open in \(G\), and choose an open set \(V\) in \(L\) such that \(U=V\cap G\). Since \(K\subseteq G\),
\[
        \pi(U)=\pi(V)\cap\pi(G).
\]
The quotient map \(\pi\) is open, so \(\pi(U)\) is open in \(\pi(G)\). Hence the natural homomorphism is open and therefore a topological isomorphism.
\end{proof}

\begin{proposition}\label{prop:general-criterion}
Let \(N\) be a closed normal subgroup of a topological group \(G\). Suppose that \(N\) is a closed normal subgroup of \(\rho G\) and that \((\rho G)/N\) is Raikov complete. Then the extension
\[
        \widehat q:\rho G\to\rho(G/N)
\]
of the quotient homomorphism \(q:G\to G/N\) is a quotient homomorphism with kernel \(N\) and satisfies
\[
        \widehat q^{-1}(G/N)=G.
\]
Consequently,
\[
        r_\rho(G/N)=\widehat q\bigl(r_\rho(G)\bigr).
\]
\end{proposition}

\begin{proof}
Let
\[
        \pi:\rho G\to (\rho G)/N
\]
be the quotient homomorphism. Since \(G\) is dense in \(\rho G\), its image \(\pi(G)\) is dense in \((\rho G)/N\). By Lemma~\ref{lem:subgroup-quotient-embedding}, the natural map from \(G/N\) onto \(\pi(G)\) is a topological isomorphism. A topological group isomorphism and its inverse are uniformly continuous for the two-sided uniformities. Thus this map is a uniform isomorphism. Since \((\rho G)/N\) is Raikov complete by hypothesis, it is a Raikov completion of \(G/N\).

Uniqueness of the Raikov completion gives a topological isomorphism
\[
        \Phi:(\rho G)/N\longrightarrow\rho(G/N)
\]
whose restriction to \(\pi(G)\) is the natural isomorphism onto \(G/N\), viewed as a subgroup of \(\rho(G/N)\). Both \(\Phi\circ\pi\) and \(\widehat q\) extend \(q\). They therefore agree on the dense subgroup \(G\) and hence on all of \(\rho G\). Since \(\pi\) is a quotient homomorphism with kernel \(N\) and \(\Phi\) is a topological isomorphism, \(\widehat q=\Phi\circ\pi\) is a quotient homomorphism with kernel \(N\).

It remains to show that
\[
        \widehat q^{-1}(G/N)=G.
\]
The restriction of \(\Phi\) maps \(\pi(G)\) onto \(G/N\), so
\[
        \Phi^{-1}(G/N)=\pi(G).
\]
If \(x\in\rho G\) and \(\widehat q(x)\in G/N\), then
\(\pi(x)\in\pi(G)\). Thus \(\pi(x)=\pi(g)\) for some \(g\in G\). Hence \(g^{-1}x\in N\subseteq G\), and therefore \(x\in G\). Conversely, since \(\widehat q\) extends \(q\), one has \(\widehat q(x)=q(x)\in G/N\) for every \(x\in G\). Thus
\[
        \widehat q^{-1}(G/N)=G.
\]
The remainder formula follows from Proposition~\ref{prop:quotient-criterion}.
\end{proof}

For a Raikov-complete normal subgroup \(N\) of \(G\), Lemma~\ref{lem:complete-kernel-closed-in-completion} supplies the closedness and normality in \(\rho G\) required above. The remaining issue is the Raikov completeness of \((\rho G)/N\).

\subsection{Raikov-complete almost metrizable kernels}\label{sec:am-kernels}

For almost metrizable kernels, the remaining hypothesis follows from known quotient-completeness results. No precompactness assumption is imposed on \(G\), and the kernel need not be compact.

Leischner's theorem concerns neutral subgroups and completeness for the supremum of the induced left and right quotient uniformities. The next lemma applies it to normal almost metrizable kernels, together with the characterization of \v{C}ech-complete almost metrizable groups. Quotients by arbitrary closed normal subgroups need not be Raikov complete \cite{Leischner1991,Leischner1993}. Related local properties of quotient maps by locally compact subgroups were studied by Arhangel'skii \cite{Arhangelskii2005}.

\begin{lemma}\label{lem:am-quotient-complete}
Let \(L\) be a Raikov-complete topological group, and let \(K\) be a closed normal subgroup of \(L\) that is both Raikov complete and almost metrizable. Then \(L/K\) is Raikov complete.
\end{lemma}

\begin{proof}
Since \(K\) is almost metrizable and Raikov complete, it is \v{C}ech-complete \cite[Theorem~4.3.15]{ArhangelskiiTkachenko2008}. As a normal subgroup, \(K\) is neutral. Leischner's theorem shows that \(L/K\) is complete for the supremum of the induced left and right quotient uniformities \cite[statement~(***)]{Leischner1993}. Because \(K\) is normal, this supremum is the two-sided uniformity of \(L/K\). Hence \(L/K\) is Raikov complete.
\end{proof}

By \cite[Lemma~13.13, p.~242]{RoelckeDierolf1981}, every locally compact topological group is almost metrizable in the sense used here. Thus locally compact kernels satisfy the almost-metrizability hypothesis in Lemma~\ref{lem:am-quotient-complete}. We shall also use the following standard completeness theorem.

\begin{lemma}[{\cite[Theorem~3.6.24]{ArhangelskiiTkachenko2008}}]\label{lem:lc-complete}
Every locally compact topological group is Raikov complete.
\end{lemma}

Quotient completeness now follows from the preceding lemmas. Propositions~\ref{prop:general-criterion} and~\ref{prop:quotient-criterion} then give the following theorem on Raikov remainders.

\begin{theorem}\label{thm:complete-am-kernel}
Let \(G\) be a topological group, and let \(N\) be a closed normal subgroup of \(G\). Suppose that \(N\) is Raikov complete and almost metrizable. Then the extension
\[
        \widehat q:\rho G\to\rho(G/N)
\]
of the quotient homomorphism \(q:G\to G/N\) is a quotient homomorphism with kernel \(N\) and satisfies
\[
        \widehat q^{-1}(G/N)=G.
\]
Consequently,
\[
        r_\rho(G/N)=\widehat q(r_\rho(G)).
\]
In particular, if \(r_\rho(G)\) is pseudocompact, then \(r_\rho(G/N)\) is pseudocompact.
\end{theorem}

\begin{proof}
By Lemma~\ref{lem:complete-kernel-closed-in-completion}, applied with \(L=\rho G\), the subgroup \(N\) is closed and normal in \(\rho G\). The topology and the two-sided uniformity on \(N\) induced from \(\rho G\) agree with those induced from \(G\). Hence \(N\) remains Raikov complete and almost metrizable. Lemma~\ref{lem:am-quotient-complete} shows that \((\rho G)/N\) is Raikov complete. Proposition~\ref{prop:general-criterion} shows that \(\widehat q\) is a quotient homomorphism with kernel \(N\) and gives the inverse-image equality and the remainder formula. The final assertion follows from Lemma~\ref{lem:continuous-image}.
\end{proof}

\begin{corollary}\label{cor:lc-kernel}
Let \(G\) be a topological group, let \(N\) be a closed locally compact normal subgroup of \(G\), and let
\[
        \widehat q:\rho G\to\rho(G/N)
\]
be the extension of the quotient homomorphism \(q:G\to G/N\). Then \(\widehat q\) is a quotient homomorphism with kernel \(N\),
\[
        \widehat q^{-1}(G/N)=G,
\]
and
\[
        r_\rho(G/N)=\widehat q\bigl(r_\rho(G)\bigr).
\]
In particular, if \(r_\rho(G)\) is pseudocompact, then \(r_\rho(G/N)\) is pseudocompact.
\end{corollary}

\begin{proof}
By Lemma~\ref{lem:lc-complete}, the group \(N\) is Raikov complete. It is also almost metrizable \cite[Lemma~13.13, p.~242]{RoelckeDierolf1981}. The conclusion follows from Theorem~\ref{thm:complete-am-kernel}.
\end{proof}

\begin{corollary}
Let \(G\) be a topological group, and let \(N\) be a closed locally compact normal subgroup of \(G\). If \(r_\rho(G)\) is compact or countably compact, then \(r_\rho(G/N)\) has the same property.
\end{corollary}

\begin{proof}
By Corollary~\ref{cor:lc-kernel},
\[
        r_\rho(G/N)=\widehat q(r_\rho(G)).
\]
The conclusion follows because compactness and countable compactness are preserved by continuous images.
\end{proof}

\subsection{Applications and examples}\label{sec:applications}

The theorem can also be stated in exact-sequence form.

\begin{corollary}\label{cor:exact-sequence}
Let
\[
        1\longrightarrow N\longrightarrow G\stackrel{q}{\longrightarrow} Q\longrightarrow 1
\]
be an exact sequence of topological groups, where \(q\) is a quotient homomorphism and \(N\) is a closed normal subgroup of \(G\). Suppose that \(N\) is Raikov complete and almost metrizable. Then the extension \(\widehat q:\rho G\to\rho Q\) of \(q\) is a quotient homomorphism with kernel \(N\), so
\[
        1\longrightarrow N\longrightarrow\rho G
        \stackrel{\widehat q}{\longrightarrow}\rho Q
        \longrightarrow 1
\]
is exact. Moreover,
\[
        \widehat q^{-1}(Q)=G.
\]
Consequently,
\[
        r_\rho(Q)=\widehat q\bigl(r_\rho(G)\bigr).
\]
In particular, if \(r_\rho(G)\) is pseudocompact, then \(r_\rho(Q)\) is pseudocompact.
\end{corollary}

\begin{proof}
Let \(q_N:G\to G/N\) be the quotient homomorphism, and write \(\widehat q_N:\rho G\to\rho(G/N)\) for its extension. Let \(\theta:G/N\to Q\) be the topological isomorphism induced by \(q\). Applying Lemma~\ref{lem:extension-to-completions} to \(\theta\) and \(\theta^{-1}\) shows that \(\theta\) extends to a topological isomorphism \(\widehat\theta:\rho(G/N)\to\rho Q\), which carries \(G/N\) onto \(Q\) and \(r_\rho(G/N)\) onto \(r_\rho(Q)\). Since \(q=\theta\circ q_N\), uniqueness of extensions gives \(\widehat q=\widehat\theta\circ\widehat q_N\). The result now follows from Theorem~\ref{thm:complete-am-kernel}.
\end{proof}

The exact-sequence formulation includes split projections. Let \(H\) be a topological group, let \(L\) be a locally compact topological group, and let
\[
        q:H\times L\longrightarrow H
\]
be the projection onto the first factor. Let
\[
        \widehat q:\rho(H\times L)\longrightarrow\rho H
\]
be its extension. The kernel of \(q\) is \(\{e_H\}\times L\), which is closed and topologically isomorphic to \(L\). Hence it is Raikov complete by Lemma~\ref{lem:lc-complete} and almost metrizable by \cite[Lemma~13.13, p.~242]{RoelckeDierolf1981}. Corollary~\ref{cor:exact-sequence} gives
\[
        r_\rho(H)=\widehat q\bigl(r_\rho(H\times L)\bigr),
\]
in agreement with Proposition~\ref{prop:product-formula}. Taking \(L=\mathbb Z\) with the discrete topology gives a noncompact discrete kernel. If \(H\) is not Raikov complete, then both \(r_\rho(H)\) and \(r_\rho(H\times L)\) are nonempty.

\subsection{Sharpness and open problems}\label{sec:sharpness}

The quotient-completeness hypothesis in Proposition~\ref{prop:general-criterion} cannot be omitted. Quotients of Raikov-complete groups need not be Raikov complete \cite{Leischner1991}.

\begin{proposition}\label{prop:not-onto}
Let \(L\) be a Raikov-complete topological group, let \(N\) be a closed normal subgroup of \(L\), and let \(q:L\to L/N\) be the quotient homomorphism. If \(L/N\) is not Raikov complete, then the extension
\[
        \widehat q:\rho L\to \rho(L/N)
\]
of \(q\) is not onto, while
\[
        \widehat q^{-1}(L/N)=\rho L=L.
\]
\end{proposition}

\begin{proof}
Since \(L\) is Raikov complete, we regard \(\rho L\) as \(L\). For every \(x\in\rho L=L\), the value \(\widehat q(x)=q(x)\) belongs to \(L/N\). Hence \(\widehat q^{-1}(L/N)=\rho L=L\). If \(\widehat q\) were onto, then
\[
        \rho(L/N)=\widehat q(\rho L)=q(L)=L/N,
\]
contrary to the assumption that \(L/N\) is not Raikov complete.
\end{proof}

Thus surjectivity of
\[
        \widehat q:\rho L\to \rho(L/N)
\]
can fail even when the preimage condition holds. In particular, \(\widehat q\) need not be onto even though \(q:L\to L/N\) is onto.

\begin{corollary}
Let \(L\), \(N\), and \(q\) be as in Proposition~\ref{prop:not-onto}. Then
\[
        r_\rho(L)=\varnothing,
\]
whereas
\[
        r_\rho(L/N)\ne\varnothing.
\]
Moreover,
\[
        \widehat q\bigl(r_\rho(L)\bigr)=\varnothing.
\]
Consequently,
\[
        r_\rho(L/N)\ne\widehat q\bigl(r_\rho(L)\bigr).
\]
\end{corollary}

\begin{proof}
Since \(L\) is Raikov complete, \(r_\rho(L)=\varnothing\). Since \(L/N\) is not Raikov complete, its Raikov completion is a proper completion and therefore \(r_\rho(L/N)\ne\varnothing\). Finally, \(\widehat q(r_\rho(L))=\widehat q(\varnothing)=\varnothing\), so the remainder formula fails.
\end{proof}

This does not yet give a counterexample to preservation of pseudocompactness under quotients. The original remainder is empty and hence pseudocompact, but the argument shows only that the quotient remainder is nonempty, not that it fails to be pseudocompact.

The next example separates surjectivity from the preimage condition of Proposition~\ref{prop:quotient-criterion}.

\begin{example}\label{ex:padic-preimage-failure}
Let \(p\) be a prime and let \(G=\mathbb Z\) be the additive group of integers with the \(p\)-adic topology. Then the Raikov completion of \(G\) is the additive group \(\mathbb Z_p\) of \(p\)-adic integers. Fix a positive integer \(k\) and put
\[
        N=p^k\mathbb Z.
\]
The subgroup \(N\) is closed and normal in \(G\), and it is almost metrizable because it is metrizable. It is not Raikov complete, since its Raikov completion is \(p^k\mathbb Z_p\).

The quotient \(G/N\) is the finite discrete group \(\mathbb Z/p^k\mathbb Z\), and hence
\[
        \rho(G/N)=G/N.
\]
The extension of the quotient homomorphism is the natural map
\[
        \widehat q:\mathbb Z_p\to \mathbb Z/p^k\mathbb Z.
\]
This map is onto. Since \(\rho(G/N)=G/N\), we have
\[
        \widehat q^{-1}(G/N)=\mathbb Z_p\ne\mathbb Z=G.
\]
Thus the preimage condition in Proposition~\ref{prop:quotient-criterion} is not automatic, even when the induced map on completions is onto.

Each fibre \(a+p^k\mathbb Z_p\) of \(\widehat q\), where \(a\in\mathbb Z\), contains points outside \(\mathbb Z\), because \(a+p^k\mathbb Z\) is a proper dense subset of \(a+p^k\mathbb Z_p\). Therefore
\[
        \widehat q(r_\rho(G))=G/N.
\]
On the other hand,
\[
        r_\rho(G/N)=\varnothing.
\]
Hence the image formula can fail when the kernel is almost metrizable but not Raikov complete. Theorem~\ref{thm:complete-am-kernel} does not apply because \(N\) is not Raikov complete.
\end{example}

Proposition~\ref{prop:not-onto} shows that surjectivity may fail while the preimage condition holds, whereas Example~\ref{ex:padic-preimage-failure} shows that the preimage condition may fail while surjectivity holds. Thus neither condition in Proposition~\ref{prop:quotient-criterion} follows from the other under the standing assumptions. The first question concerns quotient completeness beyond the almost metrizable case. Because \(N\) is normal and hence neutral, Leischner's inframetrizable quotient theorem gives completeness of \((\rho G)/N\) for the supremum of the induced left and right quotient uniformities whenever the ambient Raikov-complete group \(\rho G\) is inframetrizable. Since \(N\) is normal, this supremum is the two-sided uniformity of the quotient group. Thus \((\rho G)/N\) is Raikov complete in this case \cite{Leischner1993}. The next two questions concern preservation of pseudocompactness under quotient homomorphisms.

\begin{question}
Let \(N\) be a closed normal subgroup of a topological group \(G\), and suppose that \(N\) is Raikov complete but not almost metrizable. Outside the case where \(\rho G\) is inframetrizable, what additional conditions on \(\rho G\) and \(N\) ensure that \((\rho G)/N\) is Raikov complete?
\end{question}

\begin{question}\label{q:pseudo-counterexample}
Does there exist a topological group \(G\) and a closed normal subgroup \(N\) such that \(r_\rho(G)\) is pseudocompact but \(r_\rho(G/N)\) is not pseudocompact?
\end{question}

\begin{question}
Let \(G\) be a topological group such that \(r_\rho(G)\) is normal as a topological space and pseudocompact. For which closed normal subgroups \(N\) of \(G\) must \(r_\rho(G/N)\) be pseudocompact?
\end{question}

An affirmative answer to Question~\ref{q:pseudo-counterexample} would show that pseudocompactness of Raikov remainders is not preserved by arbitrary quotient homomorphisms. A negative answer would establish such preservation for every closed normal subgroup, without any almost-metrizability assumption on the kernel.

\end{document}